\documentclass[10pt,a4paper]{article}
\usepackage[utf8]{inputenc}
\usepackage[english]{babel}
\usepackage{amsmath}
\usepackage{amsfonts}
\usepackage{amssymb}
\usepackage{amsthm}
\usepackage{graphicx}
\usepackage{epsfig}

\usepackage{ifthen}

\usepackage{cite}
\usepackage[pagewise]{lineno} 
\usepackage{dsfont}
\usepackage[all]{xy}
\usepackage{indentfirst}

\newtheorem{thm}{Theorem}[section]
\newtheorem{cor}[thm]{Corollary}
\newtheorem{prop}[thm]{Proposition}
\newtheorem{lem}[thm]{Lemma}
\theoremstyle{remark}
\newtheorem{rem}[thm]{Remark}
\theoremstyle{definition}

\numberwithin{equation}{section}
\numberwithin{thm}{section}

\numberwithin{equation}{section}
\numberwithin{thm}{section}

\newcommand{\tr}{\mathrm{\,trace}\,}

\newcommand{\grad}{\mathrm{\,grad}}

\newcommand{\spa}{\mathrm{span\,}}

\title{Biconservative quasi-minimal immersions into semi-Euclidean spaces}
\date{}
\author{R. Ye\u gin \c Sen,  A. Kelleci, N. C. Turgay and E. \" Ozkara Canfes
}

\begin{document}

\maketitle

\begin{abstract} 
In this paper we study biconservative immersions into the semi-Riemannian space form $R^4_2(c)$ of dimension 4, index 2 and constant curvature, where $c\in\{0,-1,1\}$. First, we obtain a characterization of quasi-minimal proper biconservative immersions into  $R^4_2(c)$. Then we obtain the complete classification of   quasi-minimal  biconservative surfaces in $R^4_2(0)=\mathbb E^4_2$. We also obtain a new class of biharmonic quasi-minimal isometric immersion into $\mathbb E^4_2$.
\end{abstract}

\section{Introduction}

Surfaces of semi-Riemannian manifolds with zero mean curvature is one of mostly interested topics in differential geometry. When the ambient manifold $(N,\tilde g)$ is Riemannian, a surface with zero mean curvature, called minimal-surface, arises as the solution of the variational problem of finding the surface in $N$ with minimum area among all surfaces with the common boundary. On the other hand, if $N$ is a semi-Riemannian manifold with positive index, it admits an important class of surfaces whose mean curvature is zero. These surfaces are called \textit{quasi-minimal} surfaces and they have no counter part on Riemannian manifolds: By the definition, a submanifold $M$ of $(N,\tilde g)$ is said to be quasi-minimal if its mean curvature vector is light-like at every point. Quasi-minimal submanifolds play some fundamental roles in geometry as well as in physics and they are also called as `\textit{marginally trapped}' in the physic literature when the ambient manifold is a Lorentzian space-time, \cite{Penrose1965}.

Consider the bienergy integral 
\begin{equation}\label{BiEnergyInt}
E_2(\psi)=\frac{1}{2}\int_M\|\tau(\psi)\|^2 v_g
\end{equation}
for a mapping $\psi:(\Omega, g) \to (N, \tilde g) $ between two semi-Riemannian manifolds, where $v_g$ is the volume element of $g$ and $\tau( \psi)=-\tr \nabla d\psi$ is the tension of $\psi$. Let  $\tau_2(\psi)$ stand for the bitension field of $\psi $ defined by
$$
\tau_2(\psi)=-\Delta\tau(\psi)-\tr\left(\tilde R(d\psi, \tau(\psi))d\psi\right),
$$
where   $\Delta$ is the rough Laplacian defined on sections of $\psi^{-1}(T N)$, i.e., 
$$\Delta=-\tr \left(\nabla^\psi\nabla^\psi-\nabla^\psi_\nabla\right)$$
and $\tilde R$ is the curvature tensor of $(N, \tilde g)$. 

When \eqref{BiEnergyInt} is assumed to define a functional from $C^\infty(\Omega,N)$, it is named as bi-energy functional. In this case, the critical points of $E_2$ are called as biharmonic maps, \cite{ES}.  
%Biharmonic maps are a natural generilization of harmonic maps, critical points of the energy functional
%$$E:C^\infty(M,N)\to\mathbb R,\quad E(\psi)=\int_M\|\dd f\|^2 v_g$$
%(See \cite{BE}). 
In \cite{Ji2, Ji}, Jiang obtained the first and second variational formulas for $E_2$ and proved that $\psi$ is biharmonic if and only if the fourth order system of partial differential equations given by 
\begin{equation}\label{BihMapDef}
\tau_2(\psi)=0
\end{equation}
is satisfied. Biharmonic immersions particularly take interest of many geometers, \cite{ChenRapor,FNOBic3DSF,YuFu2014ThreeD}. 

On the other hand, if $\psi:M \to (N, \tilde g) $  is a given smooth mapping, one can also define a functional from the set of all  metrics on $M$ by using \eqref{BiEnergyInt}, \cite{FNOBic3DSF}. $(\Omega,g)$ is said to be a biconservative submanifold if $g$ is a critical point of this functional and $\psi:(\Omega,g) \hookrightarrow (N, \tilde g) $ is an isometric immersion. Note that critical points of this functional is characterized by the equation  
\begin{equation}\label{BicMapDef}
\langle \tau_2(\psi),d\psi\rangle=0,
\end{equation}
\cite{FNOBic3DSF} (See also \cite{Hilb1924}).

It is obvious that any biharmonic immersion is also biconservative. Because of this reason, biconservative submanifolds have been studied in many papers so far to understand geometry of biharmonic immersions, \cite{YuFu2014Mink3,YuFu2014Lor3S,FNOBic3DSF,TurgayHHypers}. In \cite{TurgayHHypers}, the third named author studied biconservative hypersurfaces in Euclidean spaces with three distinct principal curvatures. Also, classification results on biconservative hypersurfaces in $3$-dimensional semi-Riemannian space forms have been appeared in some papers, \cite{YuFu2014Mink3,YuFu2014Lor3S}. Most recently, biconservative surfaces in $4$-dimensional Euclidean space have been studied in \cite{YeginTurgay} and \cite{MOR2016JGA}.

In \cite{ChenE42Flat,ChenIshi1998}, all flat biharmonic quasi-minimal surfaces in the 4-dimensional pseudo-Euclidean space $ \mathbb E^4_2$ with neutral metric were obtained. Furthermore, in \cite{ChenE42Flat} the complete classification of flat quasi-minimal surfaces is given. Moreover, Chen and Garay studied quasi-minimal surfaces with parallel mean curvature vector in the pseudo-Euclidean space $ \mathbb E^4_2$  in  \cite{ChenE42PMCV}.  In this paper, we study quasi-minimal biconservative immersions into $\mathbb E^4_2$ and complete the study of biconservative quasi-minimal surfaces initiated in \cite{ChenE42Flat,ChenE42PMCV,ChenIshi1998}. In Sect. 2, we give basic definitions and equations on isometric immersions into semi-Riemannian space forms after we describe the notation used in the paper. In Sect. 3, we obtained a characterization of biconservative immersions into space forms of index 2. Finally in Sect. 4, we obtain our main result which is the complete local classification of biconservative surfaces of $\mathbb E^4_2$.

\section{Preliminaries}

We are going to denote the $n$-dimensional semi-Riemannian space form of index $s$ and constant curvature $c\in\{-1,0,1\}$ by  $R^n_s(c)$, i.e., 
$$
R^n_s(c)=\left\{
\begin{array}{cl}
\mathbb S^n_s &\mbox{if $c =1$,}\\
\mathbb E^n_s &\mbox{if $c=0$,}\\
\mathbb H^n_s &\mbox{if $c=-1$}
\end{array}
\right.
$$
and $\langle \cdot\,, \cdot\rangle$ stand for its metric tensor. When $c=0$, we define the light-cone of $\mathbb E^n_s$ by
$$\mathcal LC=\{p\in\mathbb E^n_s| \langle p,p\rangle=0 \}.$$

On the other hand, a non-zero vector $w$ in a finite dimensional non-degenerated inner product space $W$  is said to be space-like, light-like or time-like  if $\langle w,w\rangle>0$, $\langle w,w\rangle=0$ or $\langle w,w\rangle<0$, respectively. We are going to use the following well-known lemma later (see, for example, \cite[Lemma 22, p. 49]{ONeillKitap})
\begin{lem}\label{SubspaceInnProdWellKnowLemma1}\cite{ONeillKitap}
Let $V$ be a subspace of $W$ and $V^\perp$ its orthogonal complement. Then,
$\mathrm{dim\,} V +\mathrm{dim\,} V^\perp=\mathrm{dim\,} W$.
\end{lem}

Consider an isometric immersion  $f: (\Omega,g)\hookrightarrow  R^n_s(c)$  from  an $m$-dimensional semi-Riemannian manifold $(\Omega,g)$ with the Levi-Civita connection $\nabla$. Let $T\Omega$ and  $N^f\Omega$ stand for the tangent bundle of $\Omega$ and the normal bundle of $f$, respectively. If  $\widetilde{\nabla}$ denote the Levi-Civita connection of   $R^n_s(c)$, then the Gauss and Weingarten formulas are given, respectively, by
\begin{eqnarray}
\label{MEtomGauss} \widetilde\nabla_X Y&=& \nabla_X Y + \alpha_f(X,Y),\\
\label{MEtomWeingarten} \widetilde\nabla_X \xi&=& -A^f_\xi(X)+\nabla^\perp_X \xi,
\end{eqnarray}
 for any vector fields $X,\ Y\in T\Omega$ and $\xi\in N^f\Omega$, where $\alpha_f$ and  $\nabla^\perp$  are the second fundamental form and the normal connection of $f$, respectively,   and $A^f_\xi$ stands for the shape operator of $f$ along the  normal direction $\xi$. $A^f$ and $\alpha_f$  are related by
\begin{eqnarray}
\label{MinkAhhRelatedby} \langle A^f_\xi X,Y\rangle&=&\langle \alpha_f(X,Y),\xi\rangle.
\end{eqnarray}

On the other hand, the second fundamental form $\alpha_f$ of $f$,  the curvature tensor $R$ of $(\Omega,g)$    and the normal curvature tensor $R^\perp$ of $f$ satisfies the integrability conditions
\begin{subequations}
\begin{eqnarray}
\label{MinkGaussEquation} R(X,Y)Z&=&c(X\wedge Y)Z+A^f_{ \alpha_f(Y,Z)}X-A^f_{\alpha_f(X,Z)}Y,\\
\label{MinkCodazzi} (\bar \nabla_X \alpha_f )(Y,Z)&=&(\bar \nabla_Y \alpha_f )(X,Z),\\
\label{MinkRicciEquation} R^{\perp}(X,Y)\xi&=&\alpha_f(X,A^f_\xi Y)-\alpha_f(A^f_\xi X,Y),
\end{eqnarray}
\end{subequations}
called Gauss, Codazzi and Ricci equations, respectively, where, by the definition, we have
\begin{eqnarray*}
(X\wedge Y)Z&=&\langle Y,Z\rangle X-\langle X,Z\rangle Y,\\
(\bar \nabla_X \alpha_f)(Y,Z)&=&\nabla^\perp_X \alpha_f(Y,Z)-\alpha_f(\nabla_X Y,Z)-\alpha_f(Y,\nabla_X Z).
\end{eqnarray*}

The mean curvature vector field of the isometric immersion $f$  is defined by
\begin{equation}
\label{MeanCurvVectFirstDef} H^f=\frac 1m\tr\alpha_f.
\end{equation}
$f$ is said to be  quasi-minimal  if $H^f$ is light-like at every point of $\Omega$, i.e, $\langle H^f,H^f\rangle=0$ and $H^f\neq0$. In this case, $M=f(\Omega)$ is called a quasi-minimal submanifold (quasi-minimal surface if $m=2$) of $R^n_s(c)$.

Further, we are going to denote the kernel of the shape operator along $H^f$ by $T^f$, i.e., 
$$T^f=\{X\in TM| A^f_{H^f}(X)=0\}.$$

%%%%%%%%%%%%
%%%%%%%%%%%%
%%%%%%%%%%%%

\subsection{Lorentzian surfaces in $R^4_2(c)$}\label{SectLorentzQuasi}
Let  $(\Omega,g)$ be a 2-dimensional semi-Riemannian manifold. Consider an isometric immersion $f:(\Omega,g) \hookrightarrow R^4_2(c)$ and let the surface $M$ be the image of $f$, i.e., $M=f(\Omega)$. Then, the Gaussian curvature $K$  of $\Omega$ is defined by 
\begin{eqnarray}
\label{GaussianCurvature}K&=& \frac{R(X,Y,Y,X)}{\langle f_*X,f_*X\rangle\langle f_*Y,f_*Y\rangle-\langle f_*X,f_*Y\rangle^2},
\end{eqnarray}
where $X$ and $Y$ span the tangent bundle of $\Omega$. $\Omega$, and thus $M$, is said to be flat if $K$ vanishes identically.

If $g$ has index 1, then $M$ is said to be a Lorentzian surface. In this case, for any $m\in \Omega$  there exists a local coordinates system $(\mathcal N_m,(u_1,u_2))$, called isothermal coordinate system of $\Omega$, such that $m\in \mathcal N_m$ and 
$$\left.g\right|_{\mathcal N_m}=\tilde m^2(u,v)(du_1 \otimes du_1 -du_2 \otimes du_2)$$
for a positive function $\tilde m\in C^\infty(\Omega)$.  By defining a new local coordinate system $(u,v)$ by $u=\frac{u_1+u_2}{\sqrt 2}$ and $v=\frac{u_1-u_2}{\sqrt 2}$, we obtain (\cite{Chen2})
$$\left.g\right|_{\mathcal N_m}=-\tilde m^2(u,v)(du \otimes dv +dv \otimes du).$$
It is well-known that a light-like vector $w$ tangent to $M$ is propositonal to either $f_u=df(\partial_u)$ or $f_v=df(\partial_v)$.

Note that the light-like curves $u=\mbox{const}$ and $v=\mbox{const}$ are pre-geodesics of $M$. In other words, there exists a re-parametrization of the curve $u=c_1$ (or $v=c_2$) which is a geodesic of $M$. Therefore, by defining a new local coordinate system $(s,t)$  on $M$ by
$$s=s(u,v)=\int_{u_0}^u\tilde m^2(\xi,v)d\xi,\quad t=v$$
and letting $\displaystyle m(u,v)=\frac{\partial}{\partial v}\left(\int_{u_0}^u\tilde m^2(\xi,v)d\xi\right)$, 
we obtain a semi-geodesic coordinate system on $M$(see, for example, \cite{UDursun_NCT_Taiwan}).
\begin{prop}\label{SemiGeodesicCoord}
Let  $M$ be a Lorentzian surface with the metric tensor $g$. Then,  there exists a local coordinate system $(s,t)$ such that
\begin{equation}\label{DefMetricgm}
g=g_m:=-(ds \otimes dt +dt \otimes ds)+2m dt \otimes dt.
\end{equation}
Furthermore, the Levi-Civita connection of $M$ satisfies
\begin{eqnarray*}
\nabla_{\partial_s}\partial_s&=&0,\\
\nabla_{\partial_s}\partial_t= \nabla_{\partial_t}\partial_s&=& - m_s\partial_s,\\
\nabla_{\partial_t}\partial_t&=&  m_s\partial_t+(2 m m_s- m_t)\partial_s
\end{eqnarray*}
and the Gaussian curvature of $M$ is
\begin{eqnarray}\label{LorentGaussCurva}
K= m_{ss}.
\end{eqnarray}
\end{prop}

%%%%%%%%%%%%
%%%%%%%%%%%%
%%%%%%%%%%%%

\subsection{Biharmonic immersions}\label{SectBihIm}
First, we would like to recall a necessary and sufficient condition for an isometric immersion to be biharmonic. In this case by splitting $\tau_2(f)$ into its normal and tangential part and employing \eqref{BihMapDef}, one can obtain the following well-known result.
\begin{prop}\label{prop:splitting}
An isometric immersion $f:(\Omega,g)\hookrightarrow  (N,\tilde g)$  is biharmonic if and only if the equations 
\begin{equation}\label{BiconservativeEquationMostGeneral}
m\grad \left(\tilde g(H^f,H^f)\right) + 4\tr A^f_{\nabla^\perp_{\cdot}H^f}(\cdot) + 
4\tr\big(\tilde R(\cdot,H^f)\cdot\big)^T=0
\end{equation}
and
\begin{equation}\label{BiharmonicEquationMostGeneral}
\tr\alpha_f(A^f_{H^f}(\cdot),\cdot) - \Delta^\perp H^f + 
2\tr\big(\tilde R(\cdot,H^f)\cdot\big)^\perp=0
\end{equation}
are satisfied, where $m$ is the dimension of $\Omega$, $\Delta^\perp $  denote the Laplace operator associated with the normal connection of $f$. 
\end{prop}

On the other hand, if $\psi=f$ is an isometric immersion, then \eqref{BicMapDef} is equivalent to $\left(\tau_2(f)\right)^T=0$. Therefore, by using Proposition \ref{prop:splitting} we have
\begin{prop}\label{LemmaIFFbic}
An isometric immersion $f:(\Omega,g)\hookrightarrow  (N,\tilde g)$ between semi-Riemannian manifolds is biconservative if and only if  the equation \eqref{BiconservativeEquationMostGeneral} is satisfied.
\end{prop}
We immediately have the following result of Propositon \ref{LemmaIFFbic} for the case $(N,\tilde g)= R^n_s(c)$.
\begin{cor}\label{TrivialCorollar}
An isometric immersion $f:(\Omega,g)\hookrightarrow  R^n_s(c)$ is biconservative if its mean curvature vector is parallel on the normal bundle.
\end{cor}

\begin{rem}\label{REME42_MCV}
Because of Corollary \ref{TrivialCorollar}, we are going to call a biconservative isometric immersion $f$ from  $(\Omega,g)$ into $R^n_s(c)$ as \textit{proper} if $\nabla^\perp H^f\neq0$ at any point of $\Omega$. Moreover, we would like to refer to \cite{ChenE42PMCV} for classification of quasi-minimal surfaces with parallel mean curvature vector in $\mathbb E^4_2$ (See also \cite{FuHouPMCV}). 
\end{rem}

\section{Biconservative Immersions into Space Forms of Index 2}
In this section, we consider quasi-minimal biconservative immersions into $R^4_2(c)$ for $c\in\{-1,0,1\}$. Consider a 2-dimensional semi-Riemannian manifold $(\Omega,g)$, where $g$ is a Lorentzian metric. Let  $f: (\Omega,g)\hookrightarrow R^4_2(c)$ be a quasi-minimal isometric immersion and put $M=f(\Omega)$. 

We choose two vector fields $e_1,e_2$ tangent to $M$ such that $\langle e_i,e_j\rangle=1-\delta_{ij},\ i,j=1,2$. Then, there exist smooth functions $\phi_1,\ \phi_2$ such that
\begin{subequations}\label{LorentR42cLeviCitia1ALL}
\begin{eqnarray}
\label{LorentR42cLeviCitia1a} \nabla_{e_i}e_1&=&\phi_i e_1,\\
\label{LorentR42cLeviCitia1b} \nabla_{e_i}e_2&=&-\phi_i e_2.
\end{eqnarray}
\end{subequations} 
Put  $e_3=-H^f\in N^f\Omega$ and let  $e_4\in N^f\Omega$  be the unique light-like vector field satisfying $\langle e_3,e_4\rangle=-1$. 
On the other hand, if we define smooth functions $h^\alpha_{ij}$ by
$$h^\alpha_{ij}=\langle \alpha_f(e_i,e_j),e_\alpha\rangle, \quad i,j,=1,2,\ \alpha=3,4, $$
then we get
\begin{subequations}\label{QuasiR42cSecFundForm1all}
\begin{eqnarray}
\label{QuasiR42cSecFundForm1a} \alpha_f(e_i,e_i)&=&-h^4_{ii}e_3-h^3_{ii}e_4,\\
\label{QuasiR42cSecFundForm1b} \alpha_f(e_1,e_2)&=&e_3,
\end{eqnarray}
\end{subequations}
where \eqref{QuasiR42cSecFundForm1b} follows from $H^f=-\alpha_f(e_1,e_2)$. Note that we also have $h^\alpha_{ij}=\langle A^f_{e_\alpha}e_i,e_j\rangle$  because of  \eqref{MinkAhhRelatedby}. Therefore, the shape operators of $f$ satisfies  
\begin{subequations}\label{QuasiR42cShapeOp1all}
\begin{eqnarray}
\label{QuasiR42cShapeOp1a} A^f_{e_3} e_1=-h^3_{11}e_2, &\quad&  A^f_{e_3} e_2=-h^3_{22}e_1,\\
\label{QuasiR42cShapeOp1b} A^f_{e_4} e_1=e_1-h^4_{11}e_2, &\quad&  A^f_{e_4} e_2=-h^4_{22}e_1+e_2.
\end{eqnarray}
\end{subequations}
On the other hand, the Laplace operator $\Delta^\perp $  associated with the normal connection of $f$ takes the form 
$$\Delta^\perp=\nabla^\perp_{e_1}\nabla^\perp_{e_2}- \nabla^\perp_{\nabla_{e_1}e_2}+\nabla^\perp_{e_2}\nabla^\perp_{e_1}- \nabla^\perp_{\nabla_{e_2}e_1}.$$

Furthermore, one can define smooth functions $\xi_1,\xi_2$ by 
\begin{equation}
\label{QuasiR42xiidef} \nabla^\perp_{e_i} e_3=\xi_ie_3 \quad\mbox{and}\quad  \nabla^\perp_{e_i} e_4=-\xi_ie_4.
\end{equation}

We obtain the following characterization of proper biconservative immersions.
\begin{prop}\label{PropRnsc1}
Let $(\Omega,g)$ be a 2-dimensional semi-Riemannian manifold and $f: (\Omega,g)\hookrightarrow R^4_2(c)$ a quasi-minimal isometric immersion. Then, $f$ is  proper biconservative if and only if for any point $p$ such that  $A^f_{H^f}(p)\neq0$, there exists a neighborhood $\mathcal N_p$ such that $T^F$ is a degenerated distribution along which $H^F$ is parallel, where $F=\left.f\right|_{\mathcal N_p}$.
\end{prop}

\begin{proof}
Since the ambient space is $R^4_2(c)$, we have $\tr\big(\tilde R(\cdot,H^f)\cdot\big)^T=0$. Furthermore, being quasi-minimal of $f$ implies
$\grad \left(\tilde g(H^f,H^f)\right)=0$. Therefore, $f$ is biconservative if and only if 
\begin{equation}\label{ProofPropRnsc1Eq1}
\tr A^f_{\nabla^\perp_{\cdot}H^f}(\cdot)=0 
\end{equation}
because of Proposition \ref{LemmaIFFbic}. Note that \eqref{ProofPropRnsc1Eq1} is equivalent to
$$A^f_{\nabla^\perp_{e_1}e_3}(e_2)+A^f_{\nabla^\perp_{e_2}e_3}(e_1)=0$$
in terms of vector fields $e_1,e_2,e_3$ defined above. By considering \eqref{QuasiR42cShapeOp1a} and \eqref{QuasiR42xiidef}, we conclude that 
 $f$ is biconservative if and only if 
\begin{equation}\label{PropRnsc1Eq1} 
\xi_1h^3_{22}e_1+\xi_2h^3_{11}e_2=0.
\end{equation}
Note that being \textit{proper } of the biconservative immersion $f$ implies $\xi_1(q)\neq 0$ or $\xi_2(q)\neq 0$ at any point $q\in \Omega$. 

Now, in order to prove the necessary condition, assume that  $f$ is a proper biconservative immersion and let $A^f_{H^f}(p)\neq0$ at a point  $p$ of $\Omega$. Then, without loss of generality, we may assume $h^3_{22}(p)\neq 0$ on a neighboorhood $\mathcal N_p$ of $\Omega$. In this case, because of \eqref{PropRnsc1Eq1}, we have $\xi_1=0$ on $\mathcal N_p$ which implies $\xi_2(q)\neq 0$ for any $q\in\mathcal N_p$. Thus, \eqref{PropRnsc1Eq1} implies $h^3_{11}=0$ on $\mathcal N_p$. Put $F=\left.f\right|_{\mathcal N_p}$. Then, we have $T^F=\spa\{e_1\}$ which  is a degenerated distribution.  Moreover, since $\xi_1$ vanishes identically on $\mathcal N_p$, we have $\nabla^\perp_X H^F=0$ whenever $X\in T^F$. Hence, we have completed the proof of the necessary condition. 

For the proof of the sufficient condition, we consider the following two cases  separately. If $A^f_{H^f}=0$, then \eqref{QuasiR42cShapeOp1a} implies $h^3_{11}=h^3_{22}=0$. Therefore \eqref{PropRnsc1Eq1} is satisfied.  On the other hand, consider the case $A^f_{H^f}(p)\neq0$ on $\mathcal N_p$. Assume $T^F=\spa \{e_1\}$ for a light-like vector field $e_1$ and let $\nabla^\perp_{e_1} e_3=0$. Then, we have $h^3_{11}=\xi_1=0$. Therefore \eqref{PropRnsc1Eq1} is  satisfied again. Hence, the proof of the sufficient condition is completed.\end{proof}

Now, we study the case $c=0$. Let $f: (\Omega,g_m)\hookrightarrow \mathbb E^4_2$ be a proper biconservative quasi-minimal immersion, where $\Omega=I\times J$ and $g_m$ is the metric defined by  \eqref{DefMetricgm} for a $m\in C^\infty(\Omega)$. Assume that the Gaussian curvature $K$ of $\Omega$ does not vanish. Note that, because of the Gauss equation \eqref{MinkGaussEquation}, if $A^f_{H^f}=0$ at a point $p\in\Omega$, then the Gaussian curvature $K(p)=0$ which is a contradiction. Therefore, we have $A^f_{H^f}(q)\neq0$ for all  $q\in\Omega.$ Hence,  Proposition \ref{PropRnsc1} implies that  $T^f$ is a degenerated distribution along which  $H^f$ is parallel. In terms of a local pseudo-orthonormal frame field $\{e_1,e_2;e_3,e_4\}$, we have $A^f_{e_3} e_1=0$ and $\nabla^\perp_{e_1}e_3=0$, or, equivalently, $h^3_{11}=\xi_1=0$.  Now, $f$ is biharmonic if and only if  \eqref{BiharmonicEquationMostGeneral} is satisfied. However, since $\tilde R=0$, \eqref{BiharmonicEquationMostGeneral}  becomes
$$\nabla^\perp_{e_1}\nabla^\perp_{e_2}e_3- \nabla^\perp_{\nabla_{e_1}e_2}e_3 =\alpha_f(A^f_{e_3}e_2,e_1)$$
which is equivalent to the Ricci equation \eqref{MinkRicciEquation} for $X=e_1$, $Y=e_2$ and $\xi=e_3$. Hence, we have the following result.
\begin{thm}\label{ThmBICimpliesBIH}
Let $(\Omega,g)$ be a Lorentzian surface with the Gaussian curvature $K$ and $f: (\Omega,g)\hookrightarrow  \mathbb E^4_2$ a quasi-minimal isometric immersion. Assume that $K$ does not vanish. If $f$ is a proper biconservative immersion, then it is biharmonic.
\end{thm}

%%%%%%%%%%%%%%%%%%
%%%%%%%%%%%%%%%%%%%%%%%%%%%%%%%%%%%%
%%%%%%%%%%%%%%%%%%%%%%%%%%%%%%%%%%%%
%%%%%%%%%%%%%%%%%%%%%%%%%%%%%%%%%%%%
%%%%%%%%%%%%%%%%%%%%%%%%%%%%%%%%%%%%
%%%%%%%%%%%%%%%%%%%%%%%%%%%%%%%%%%%%
%%%%%%%%%%%%%%%%%%%%%%%%%%%%%%%%%%%%
%%%%%%%%%%%%%%%%%%%%%%%%%%%%%%%%%%%%
%%%%%%%%%%%%%%%%%%%%%%%%%%%%%%%%%%%%
%%%%%%%%%%%%%%%%%%

\section{Biconservative Surfaces of $\mathbb E^4_2$}
In this section, we focus on immersions into the pseudo-Euclidean space  $\mathbb E^4_2$ with neutral metric. We get the complete local classification of quasi-minimal, biconservative surfaces.

First, we consider flat surfaces and get the following classification of biconservative surfaces. We want to note that the proof of this proposition immediately follows from the proof of \cite[Theorem 4.1]{ChenE42Flat}.
\begin{prop}\label{PropE42FLAT}
A flat surface in $\mathbb E^4_2$ is quasi-minimal and biconservative   if and only if locally congruent to one of the following surfaces:
\begin{enumerate}
\item[(i)] The surface given by  $f(s,t)=(\psi(s,t),\frac{s-t}{\sqrt2},\frac{s+t}{\sqrt2},\psi(s,t)),\ (s,t)\in U$,  where $\psi:U\to\mathbb R$ is a smooth function and $U$ is open in $\mathbb R^2$,

\item[(ii)] The surface given by  $f(s,t)=z(s)t +w(s),$  where $z(s)$ is a light-like curve in the light-cone $\mathcal LC$ and $w$ is a light-like curve satisfying $\langle z',w'\rangle=0$ and $\langle z,w'\rangle=-1$.
\end{enumerate}
\end{prop}

\begin{proof}
A direct computation shows that the above surfaces are flat, quasi-minimal and biconservative. Conversely assume that $M$ is a flat quasi-minimal biconservative surface and $p\in M$. Consider a frame field $\{e_1,e_2;e_3,e_4\}$ described in Sect. 3 and let $h^\alpha_{ii}$ are functions defined by \eqref{QuasiR42cSecFundForm1all} and \eqref{QuasiR42xiidef}, respectively. Then, by  Proposition \ref{PropRnsc1}, we have two cases: $h^3_{11}=h^3_{22}=0$  and $h^3_{11}=h^4_{11}=0$. By considering the proof of \cite[Theorem 4.1]{ChenE42Flat}, one can conclude that $f$ is congruent to one of these two surfaces given in the proposition.
\end{proof}

%%%%%%%%%%%%%%%%%%%%%%%%%%%%%%%%%%%%%%%%%%%
%%%%%%%%%%%%%%%%%%%%%%%%%%%%%%%%%%%%%%%%%%%
%%%%%%%%%%%%%%%%%%%%%%%%%%%%%%%%%%%%%%%%%%%

Now, we are going to consider non-flat quasi-minimal surfaces with non-parallel mean curvature vector. First we define an intrinsic $L:\Omega\to\mathbb R$ of $(\Omega,g_m)$ by 
 \begin{equation}\label{DefinitionofL}
L= -\frac{K_t+mK_s+3m_sK}{K}.
\end{equation}
Then, we construct the following example of biconservative immersion from a non-flat two-dimensional Lorentzian manifold into $\mathbb E^4_2$. We would like to note that this immersion is also biharmonic because of Theorem \ref{ThmBICimpliesBIH}.

\begin{prop}\label{PropExample}
Let $\Omega=I \times J$ for some open intervals $I,J$ and $m\in C^\infty(\Omega)$ and assume that the intrinsic $L:\Omega\to\mathbb R$ of $(\Omega,g_m)$ satisfies $L=L(t)$. Consider a light-like curve $\alpha:J\hookrightarrow \mathbb E^4_2$ lying on $\mathcal LC$ such that
$V_t=\spa\{\alpha(t),\alpha'(t)\} $ is two dimensional for all $t\in J$. Assume that  $\eta:J\to\mathbb R^4$ satisfies the conditions 
\begin{subequations}\label{PropExampleStatementEq2}
\begin{eqnarray}
\label{PropExampleStatementEq2a}\langle \eta' ,\eta' \rangle&=&0,\\
\label{PropExampleStatementEq2b}\langle \alpha , \eta' \rangle&=&0,\\
\label{PropExampleStatementEq2c}\langle \eta' ,\alpha' \rangle&=&-\frac{1}{a },\\
\label{PropExampleStatementEq2d}\langle \eta' ,\alpha'' \rangle&=&\frac{2 a' -a  L }{a ^2}
\end{eqnarray}
\end{subequations}
for a function  $a\in C^\infty(J)$.  Then, the mapping
\begin{equation}\label{PropExampleStatementEq3}
\begin{array}{rcl}
f:(\Omega,g_m)&\longrightarrow& \mathbb E^4_2\\
f(s,t)&=&\eta(t)+\Big(s a'(t)-a(t) (m(s,t)+s L(t))\Big)\alpha(t)\\&&+sa(t)\alpha'(t) 
\end{array}
\end{equation}
is a quasi-minimal, proper biconservative isometric immersion.
\end{prop}

\begin{proof}
Since the light-like curve $\alpha$ lies on $\mathcal LC$, we have
\begin{equation}\label{Condsonalpha} 
\langle \alpha,\alpha\rangle=\langle \alpha',\alpha'\rangle=0
\end{equation}
which implies $\left\langle \cdot\,, \cdot\rangle\right|_{V_t}=0$. Thus, we have $V_t\subset V_t^\perp$. Therefore, Lemma \ref{SubspaceInnProdWellKnowLemma1} implies 
$V_t= V_t^\perp$. Note that \eqref{Condsonalpha} also gives $\alpha''(t)\in V_t^\perp=V_t.$ Thus, we have
\begin{equation}\label{PropExampleStatementEq4}
\alpha''(t)=A(t)\alpha(t)+\frac{a(t)L(t)-2a'(t)}{a(t)}\alpha'
\end{equation}
for a smooth function $A$ because of \eqref{PropExampleStatementEq2b}-\eqref{PropExampleStatementEq2d}. By a direct computation considering \eqref{PropExampleStatementEq3} and \eqref{PropExampleStatementEq4} we  obtain $\langle f_s,f_s\rangle=0$, $\langle f_s,f_t\rangle=-1$ and $\langle f_t,f_t\rangle=2m$ which yields that $f$ is an isometric immersion. 

By a further computation, we get
$$e_3=\alpha_f (e_1,e_2)=-H^f=B(t)\alpha(t)\quad\mbox{ and } \quad \alpha_f (e_1,e_1)=a(t)m_{ss}(s,t)\alpha(t).$$
Therefore, we have $h^3_{11}=0$ and $\nabla^\perp_{e_1}e_3=0$ which yields that $T^f=\spa\{\partial_s\}$ and $H^f$ is parallel along $T^f$.  Hence, Proposition \ref{PropRnsc1} yields that $f$ is biconservative.
\end{proof}

In the remaining part of this section we are going to show that the converse of Proposition \ref{PropExample} is also true.
\begin{lem}\label{LemLCConditions1}
 Let  $f: (\Omega,g_m)\hookrightarrow  \mathbb E^4_2$ be a  quasi-minimal immersion and assume that the Gaussian curvature of $(\Omega,g_m)$ does not vanish. Consider the pseudo-orthonormal  frame field $\{e_1,e_2;e_3,e_4\}$ such that $e_1=f_*\partial_s$, $e_2=f_*\left(m\partial_s+\partial_t\right)$ and $e_3=-H^f$.
If $f$ is  proper biconservative and $T^f=\spa\{\partial_s\}$, then the Levi Civita connection of $\mathbb E^4_2$ satisfies
\begin{subequations}\label{LCConditions1}
\begin{eqnarray}
\label{LCConditions1a} \widetilde\nabla_{ {e_1}} {e_1}=-am_{ss} {e_3},  &\qquad &\widetilde\nabla_{ {e_2}}{ {e_1}}=- m_s{ {e_1}}+ {e_3},\\
\label{LCConditions1b} \widetilde\nabla_{ {e_1}}{ {e_2}}= {e_3}, &\qquad &\widetilde\nabla_{ {e_2}} {e_2}=  m_s{ {e_2}}+(m+bs-z)  {e_3}-\frac 1{a} {e_4},\\
\label{LCConditions1c} \widetilde\nabla_{ {e_1}} {e_3}=0, &\qquad &\widetilde\nabla_{ {e_2}} {e_3}=\frac 1{a}{ {e_1}}+(m_s+b) {e_3},\\
\nonumber \widetilde\nabla_{ {e_1}} {e_4}=-{ {e_1}}+am_{ss}{ {e_2}}, &\qquad & \widetilde\nabla_{ {e_2}} {e_4}=(z-m-bs){ {e_1}}-{ {e_2}}-(m_s+b) {e_4}
\end{eqnarray} 
for some smooth functions $a,b,z$ such that $e_1(a)=e_1(b)=e_1(z)=0$ and
\begin{equation}\label{LemLCCequation}
b+\frac{a'}{a}=L.
\end{equation}
\end{subequations}
\end{lem}

\begin{proof}
 Assume that $f$ is  proper biconservative. Then, we have $h^4_{11}=\xi_1=0$ because of  Proposition \ref{PropRnsc1}.   Thus,  we have 
\begin{align}\label{LCConditions2}
\begin{split}
 \widetilde\nabla_{ {e_1}} {e_1}=\phi_1 e_1-h^4_{11} {e_3},  \qquad &\widetilde\nabla_{ {e_2}}{ {e_1}}= \phi_2{ {e_1}}+ {e_3},\\
 \widetilde\nabla_{ {e_1}}{ {e_2}}=-\phi_1 e_2+ {e_3}, \qquad &\widetilde\nabla_{ {e_2}} {e_2}= -\phi_2{ {e_2}}-h^4_{22}  {e_3}-h^3_{22} {e_4},\\
 \widetilde\nabla_{ {e_1}} {e_3}=0, \qquad &\widetilde\nabla_{ {e_2}} {e_3}=h^3_{22}{ {e_1}}+\xi_2 {e_3},\\
 \widetilde\nabla_{ {e_1}} {e_4}=-{ {e_1}}+h^4_{11}{ {e_2}}, \qquad & \widetilde\nabla_{ {e_2}} {e_4}=h^4_{22}{ {e_1}}-{ {e_2}}-\xi_2 {e_4}.
\end{split}
\end{align}
Note that Proposition \ref{SemiGeodesicCoord} implies 
\begin{equation}\label{GaussFormulaForCase}
\phi_1=0\quad\mbox{ and }\quad\phi_2=-m_s.
\end{equation}

On the other hand, Codazzi equation \eqref{MinkCodazzi} for $X=Z=e_2$, $Y=e_1$ gives
\begin{equation}
\label{Cod1ForCase} e_1(h^4_{22})=-\xi_2, \quad e_1(h^3_{22})=0, 
\end{equation}
Furthermore, by using Gauss equation \eqref{MinkGaussEquation} and Ricci equation \eqref{MinkRicciEquation},  we get
\begin{eqnarray}
\label{GaussEqForCase} h^4_{11}h^3_{22}=K=m_{ss} &\quad&\\ 
\label{RicciEqForCase} e_1(\xi_2)=K=m_{ss}. &\quad&
\end{eqnarray}
By  taking into account $e_1=f_*\partial_s$, $e_2=f_*\left(m\partial_s+\partial_t\right)$, we consider \eqref{Cod1ForCase}, \eqref{GaussEqForCase} and \eqref{RicciEqForCase} to get
\begin{align}\label{LCConditions3}
\begin{split}
h^3_{22}(s,t)=\frac 1{a(t)}, &\quad h^4_{11}(s,t)=a(t)m_{ss}(s,t),\\
h^4_{22}(s,t)=-m(s,t)-b(t)s+z(t), &\quad \xi_2(s,t)=m_s(s,t)+b(t) 
\end{split}
\end{align}
for some   $a,b,z\in C^\infty(J)$. Finally, by combining \eqref{GaussFormulaForCase} and \eqref{LCConditions3} with \eqref{LCConditions2}, we obtain \eqref{LCConditions1}.

On the other hand,  Codazzi equation \eqref{MinkCodazzi} for $X=Z=e_1$, $Y=e_2$ gives 
 \begin{equation*}
K(a'+ab)+(K_t+mK_s+3m_sK)a=0.
\end{equation*}
By combining this equation with \eqref{DefinitionofL}, we get \eqref{LemLCCequation}.
\end{proof}

%%%%%%%%%%%%%%%%%%%%%%%%%%%%%%%%%%%%%%%%%%%%%%%%%%
%%%%%%%%%%%%%%%%%%%%%%%%%%%%%%%%%%%%%%%%%%%%%%%%%%
%%%%%%%%%%%%%%%%%%%%%%%%%%%%%%%%%%%%%%%%%%%%%%%%%%
%%%%%%%%%%%%%%%%%%%%%%%%%%%%%%%%%%%%%%%%%%%%%%%%%%
%%%%%%%%%%%%%%%%%%%%%%%%%%%%%%%%%%%%%%%%%%%%%%%%%%
%%%%%%%%%%%%%%%%%%%%%%%%%%%%%%%%%%%%%%%%%%%%%%%%%%
%%%%%%%%%%%%%%%%%%%%%%%%%%%%%%%%%%%%%%%%%%%%%%%%%%
%%%%%%%%%%%%%%%%%%%%%%%%%%%%%%%%%%%%%%%%%%%%%%%%%%

Next, we get a necessary and sufficient condition for the existence of biconservative immersions from a Lorentzian surface $(\Omega,g_m)$.

\begin{prop}
Let $m\in C^\infty(\Omega)$ and $\Omega=I \times J$ for some open intervals $I,J$ and consider the Lorentzian surface $(\Omega,g_m)$ with non-vanishing Gaussian curvature, where $g_m$ is the metric defined by  \eqref{DefMetricgm}.  Then, $(\Omega,g_m)$ admits a quasi-minimal, proper  biconservative  isometric immersion with non-parallel mean curvature vector such that $T^f=\spa\{\partial_s\}$ if and only if $L=L(t)$. 
\end{prop}

\begin{proof}
In order to prove necessary condition, we assume the existence of such immersion $f$. Then, the Levi-Civita  connection $\widetilde\nabla$ satisfies \eqref{LCConditions1} because of Lemma \ref{LemLCConditions1}. \eqref{LemLCCequation} yields $\partial_s(L)=0$. Conversely, if $L=L(t)$, the immersion $f$ described by \eqref{PropExampleStatementEq3} is proper biconservative by Proposition \eqref{PropExample}. Hence, the proof is completed.
\end{proof}

\begin{thm}\label{LCTnonflat} 
If $M$ is a  proper biconservative, quasi-minimal surface with non-vanishing Gaussian curvature, then
it is locally congruent to the image $f(\Omega)$ of the isometric immersion $f$ given in Proposition \ref{PropExample}. 
\end{thm}

\begin{proof}
Let $(\Omega,g_m)$ has non-vanishing Gaussian curvature. Consider a quasi-minimal  isometric immersion  $f:(\Omega,g_m)\hookrightarrow R^4_2(c)$ and put $M=f(\Omega)$. Assume that $f$ is  proper biconservative. Then, $\widetilde\nabla$ satisfies \eqref{LCConditions1} because of Lemma \ref{LemLCConditions1}.  The first equation in \eqref{LCConditions1c} gives $\frac{\partial e_3}{\partial s}=0$ which implies
\begin{equation}\label{MnThmEq1}
e_3(s,t)=\alpha(t)
\end{equation}
for a mapping $\alpha:J\to\mathbb E^4_2$. Also the second equation in \eqref{LCConditions1c} and \eqref{MnThmEq1}  give
\begin{equation}\label{MnThmEq1.5}
\alpha'(t)=\frac 1{a(t)}f_s(s,t)+(m_s(s,t)+b(t)) \alpha(t).
\end{equation}
By considering \eqref{MnThmEq1} and \eqref{MnThmEq1.5} one can see that $\alpha$ is a light-like curve  lying on $\mathcal LC$ because $K$ does not vanish.

On the other hand, the first equation in \eqref{LCConditions1a} turns into 
$$f_{ss}(s,t)=-a(t)m_{ss}(s,t)\alpha(t)$$
whose solution is
\begin{equation}\label{MnThmEq2}
f(s,t)=-a(t) m(s,t) \alpha(t) + s\xi(t)+\eta(t)
\end{equation}
for some functions $\xi,\eta: J\to\mathbb E^4_2$. 
By using \eqref{MnThmEq2} and considering \eqref{LemLCCequation} in this equation, we obtain
\begin{equation}\label{MnThmEq3}
\xi=(a'-aL)\alpha+a \alpha'.
\end{equation}
By combining \eqref{MnThmEq2} and \eqref{MnThmEq3} we get \eqref{PropExampleStatementEq3}. 

Now, since $f$ is an isometric immersion, we have $\langle f_s,f_t\rangle=-1$ and $\langle f_t,f_t\rangle=2m$. By a direct computation using $\langle f_s,f_t\rangle=-1$ and \eqref{PropExampleStatementEq3}, we obtain
$$-am_s\langle\alpha,\eta'\rangle+ a' \langle\alpha,\eta'\rangle-a L \langle\alpha,\eta'\rangle+a \langle\alpha',\eta'\rangle=-1$$
which gives the first equation in \eqref{PropExampleStatementEq2a} and \eqref{PropExampleStatementEq2c} because $K=m_{ss}$ does not vanish.
On the other hand, $\langle f_t,f_t\rangle=2m$ and \eqref{PropExampleStatementEq3} imply
$$2 m+2 s \left(-2\frac{ a'}{a}+ L+ a \langle\alpha'',\eta'\rangle\right)+\langle \eta ',\eta '\rangle=2m$$
which gives the second equation in \eqref{PropExampleStatementEq2b} and \eqref{PropExampleStatementEq2d}. Hence $f$ is as given in Proposition \ref{PropExample} which completes the proof.
\end{proof}

By combining Proposition \ref{PropE42FLAT} and Theorem \ref{LCTnonflat}, we obtain the following complete classification of  quasi-minimal, proper biconservative surfaces in $\mathbb E^4_2$.  
\begin{thm}
A  surface $M$ in $\mathbb E^4_2$ is quasi-minimal and proper biconservative if and only if it is congruent to one of the following surfaces:
\begin{enumerate}
\item[(i)] The surface given by  $f(s,t)=(\psi(s,t),\frac{s-t}{\sqrt2},\frac{s+t}{\sqrt2},\psi(s,t)),\ (s,t)\in U$,  where $\psi:U\to\mathbb R$ is a smooth function and $U$ is open in $\mathbb R^2$,

\item[(ii)] The surface given by  $f(s,t)=z(s)t +w(s),$  where $z(s)$ is a light-like curve in the light-cone $\mathcal LC$ and $w$ is a light-like curve satisfying $\langle z',w'\rangle=0$ and $\langle z,w'\rangle=-1$,

\item[(iii)] The surface given in Proposition \ref{PropExample}. 
\end{enumerate}
\end{thm}

Finally, by combining \cite[Theorem 5.1]{ChenE42Flat} with Theorem \ref{ThmBICimpliesBIH} and Theorem \ref{LCTnonflat}, we get
\begin{thm}
A  surface $M$ in $\mathbb E^4_2$ is quasi-minimal  and biharmonic if and only if it is congruent to one of the following surfaces:
\begin{enumerate}
\item[(i)] The surface given by  $f(s,t)=(\psi(s,t),\frac{s-t}{\sqrt2},\frac{s+t}{\sqrt2},\psi(s,t)),\ (s,t)\in U$,   for a smooth function $\psi:U\to\mathbb R$ satifying $f_{st}\neq0$ and $f_{sstt}=0$, where $U$ is open in $\mathbb R^2$,

\item[(ii)] The surface given by  $f(s,t)=z(s)t +w(s),$  where $z(s)$ is a light-like curve in the light-cone $\mathcal LC$ and $w$ is a light-like curve satisfying $\langle z',w'\rangle=0$ and $\langle z,w'\rangle=-1$,

\item[(iii)] The surface given in Proposition \ref{PropExample}. 
\end{enumerate}
\end{thm}

\section*{Acknowledgements}
This work was obtained during the  ITU-GAP project \emph{ARI2Harmoni} (Project Number: TGA-2017-40722).

\bibliographystyle{amsplain}

\end{document}